\documentclass[11pt]{article}
\usepackage[slantedGreek]{mathpazo}
\usepackage{amsmath,amssymb,amsthm,amsfonts}
\usepackage[top=8pc,bottom=8pc,left=8pc,right=8pc]{geometry}

\usepackage{subcaption}
\usepackage{multirow}

\usepackage{bm}

\RequirePackage[hyphens]{url}
\RequirePackage[colorlinks,citecolor=blue,urlcolor=blue]{hyperref}
\usepackage[authoryear]{natbib}

\usepackage[plain,noend]{algorithm2e}

\newtheorem{theorem}{Theorem}

\newtheorem{lemma}[theorem]{Lemma}

\newtheorem{proposition}[theorem]{Proposition}
\newtheorem{remark}[theorem]{Remark}

\usepackage{array}
\def\CS{Cauchy-Schl\"omilch}
\def\PG{P{\'o}lya-Gamma}

\newcommand{\ben}{\begin{enumerate}}
\newcommand{\een}{\end{enumerate}}
\newcommand{\beq}{\begin{equation}}
\newcommand{\eeq}{\end{equation}}

\newcommand{\half}{\frac{1}{2}}

\newcommand{\vectornorm}[1]{\left|\left|#1\right|\right|}

\newcommand{\NormRV}{\mathcal{N}}
\newcommand{\InvGaussRV}{\mathcal{IG}}
\newcommand{\CauchyRV}{\mathcal{C}}
\newcommand{\GammaRV}{\mathcal{G}}
\newcommand{\UnifRV}{\mathcal{U}}

\begin{document}



\markboth{Bhadra, Datta, Polson, and Willard}{Global-Local Mixtures}

\title{Global-Local Mixtures}

\author{Anindya Bhadra  \footnote{email: bhadra@purdue.edu} \\ Purdue University
\and Jyotishka Datta  \footnote{email: jd298@stat.duke.edu}\\ University of Arkansas \\
\and Nicholas G. Polson \footnote{email: ngp@chicagobooth.edu} and Brandon Willard \footnote{email: brandonwillard@gmail.com} \\ The University of Chicago, Booth School of Business}

\maketitle

\begin{abstract}
  Global-local mixtures are derived from the \CS{} and Liouville integral transformation identities. We characterize well-known normal-scale mixture
  distributions including the Laplace or lasso, logit and quantile as well as new global-local mixtures. We also apply our methodology to convolutions that
  commonly arise in Bayesian inference. Finally, we conclude with a conjecture concerning bridge and uniform correlation mixtures. 
\end{abstract}


\section{Introduction}
Many statistical problems involve regularization penalties derived from
global-local mixture distributions \citep{polson_data_2011,
hans2011comment,bhadra2015horseshoe+}. A global-local mixture density, denoted
by $p(x_1, \ldots, x_p)$, takes the form 
\begin{equation*}
  p(x_1, \ldots, x_p) = \int_{0}^{\infty}\prod_{i=1}^{p} p(x_i \mid \tau) 
  p(\tau) d\tau, 
\end{equation*}
where $p(x_i \mid \tau) = \int_{0}^{\infty} p(x_i \mid \lambda_i, \tau) p(\lambda_i \mid \tau) d\lambda_i$ is a local mixture and $p(x_1, \ldots, x_p)$ is a global mixture over $\tau \sim p(\tau)$. There is great interest in analytically calculating $p(x_i \mid \tau)$, and the associated regularization penalty $\phi(x_i, \tau) = -\log p(x_i \mid \tau)$.  Convolution mixtures of the form $p(x_i \mid \tau) = \int p(x_i - \lambda_i) p(\lambda_i) d \lambda_i$ are also of interest. We show how the \CS{} and Liouville transformations can
be used to derive closed-form global-local mixtures.  We start by stating two key integral identities: 
the \CS{} transformation 
\begin{equation}
  \int_0^\infty f \left\{ ( a x - b x^{-1} )^2 \right\} d x 
  = \frac{1}{2a} \int_0^\infty f(y^2) d y
  \;, 
  \label{eq:identity}
\end{equation}
and the Liouville transformation
\begin{equation}
  \int_{0}^{\infty} f\left(ax + \frac{b}{x} \right) x^{-1/2}dx = a^{-1/2} \int_{0}^{\infty} f\left\{ 2 (ab)^{1/2} + y \right\} y^{-1/2} dy, \quad a, b >0
  \;. 
  \label{eq:liouville}
\end{equation}
See \citet{boros2006irresistible}, \citet{baker2008probabilistic} and \citet{jones_generating_2014} for further discussion. Identity \eqref{eq:identity} follows from the simple transformation $t = b/(a x)$ as
\begin{equation*}
  I = \int_{0}^{\infty} f \left\{(ax - b/x)^2 \right\} dx = \int_{0}^{\infty} f \left\{(at - b/t)^2 \right\} \frac{b}{a t^2} dt.
\end{equation*}
Adding the two terms in the last equality yields $2 I = \int_{0}^{\infty} f \left\{(at - b/t)^2 \right\} \left\{ 1+{b}/({a t^2}) \right\} dt$ and transforming $y = b/t - at$ gives $dy = -a \{1+{b}/({a t^2})\} dt$, yielding $I = (2a)^{-1} \int_{0}^{\infty} f(y^2) dy$, as required. A useful generalization of the \CS{} transformation is
\begin{equation}
  \int_0^\infty f\left[ \{x-s(x)\}^2 \right] dx =  \int_0^\infty f( y^2 ) dy, \label{eq:gen}
\end{equation}
where $s(x)=s^{-1}(x)$ is a self-inverse function such as $s(x) = b/x$ or $s(x) = -a^{-1}\log\{1-\exp(a x)\}$. The proof for the Liouville transformation identity follows in a
similar manner, and is omitted for the sake of brevity. These identities can be used to construct new global-local mixture distributions. 
Let $f(x) = 2g\{ t(x) \}$ and let $t(x)$ be of the form $x-s(x)$, where $s : \Re^+ \to \Re^+$ is a self-inverse, onto and monotone decreasing function. Together with the \CS{} transformation, we have a rather surprising way to represent the resulting $g\{t(x)\}$ as a global-local scale mixture. 

\citet{jones_generating_2014} shows that only a few choices of $t(x)$ leads to fully tractable formulae for its inverse $t^{-1}= \Pi$ and the integral 
$\Pi(y) = \int_{-\infty}^{y} \pi(\omega) d\omega$. Two special choices are the $t$-distribution with 2 degrees of freedom and the logistic. 
\begin{align*}
\Pi_{T}(y) = (1/2)\{ y+(4b+y^2)^{1/2}\}, \Pi_T^{-1}(x) = t_T(x) = x - b/x, \quad b >0,\\
\Pi_{L}(y) = a^{-1} \log(1+e^{ay}), \Pi_L^{-1}(x) = t_L(x) = a^{-1} \log(e^{ax}-1), \quad a>0.
\end{align*}

Now, the integral identity in \eqref{eq:identity} shows that if $f(x), \;x\geq 0$ is a density function, so is $g(x) = 2a f(|ax-b/x|), \; x > 0$.  The functions $f$ and $g$ are called mother and daughter density functions, respectively.  

Apart from simplifying proofs involving global-local mixtures, the \CS{} and Liouville transformations can generate new distributions via scale transformations. These transformations can take the form $f(x) = 2 g\{ t(x) \}$ for certain $f(x)$ under suitable conditions. For example, given a density $f(x)$ we can create a new global-local
scale family, $f(a x - b/x)$, by effectively reallocating its probability mass. A particularly useful tool for generating univariate and multivariate random
variables is Khintchine's theorem.  
Khintchine's theorem states that any random variable $X$ with a unimodal, univariate distribution and a mode at zero can be written as a product 
$X = Z U$, where $U \sim \UnifRV(0,1)$ and $Z$ has the density function $f_Z(z) = -z f^{\prime}_{X}(z), z \in \Re$. \citet{bryson1982constructing}, and subsequently \citet{jones2012khintchine}, discuss how Khintchine's theorem allows one to construct both univariate and multivariate densities, even with special dependence structure. \cite{jones_generating_2014} develops an extended Khintchine's theorem that further allows one to generate random variables with unimodal densities of the
form $2 g\{t(x)\}$.

\section{Global-local Scale Mixtures}
\label{sec:gls_mixes}
\subsection{Lasso as a normal scale mixture}
The lasso penalty arises as a Laplace global-local mixture
\citep{andrews_scale_1974}.  A simple transformation proof follows using \CS{}
with $f(x) = e^{-x}$.  Starting with the normal integral identity, 
$\int_{0}^{\infty} f(y^2) dy = \int_0^\infty e^{-y^2} dy = \pi^{1/2}/2 $, we
obtain
\[
\int_0^\infty e^{-(a x)^2 - (b/x)^2} d x = \int_0^{\infty} 
\exp\left\{-a b \left(\frac{a}{b} x^2 + \frac{b}{a} x^{-2} \right)\right\} dx 
= \frac{\pi^{1/2}}{2a} e^{-2 a b}
\quad a,b \in \Re.
\]
Substituting $t = (a/b)^{1/2} x$ and $c = ab$ yields the Laplace or Lasso penalty as
\begin{align*}
  \int_0^\infty e^{- c (t - t^{-1})^2} dt 
  &= \half (\pi/c)^{1/2} 
  \Rightarrow \int_0^\infty e^{- c (t^2 + t^{-2})} dt 
  = \half (\pi/c)^{1/2} e^{-2c}
  \;. 
\end{align*}
The Laplace density can be viewed as a transformed normal, via $y = t - t^{-1}$.

\begin{proposition}
The usual identity for the lasso also follows from \citet{levy1940certains} as
\begin{equation}
  \int_{0}^{\infty} \frac{a}{(2 \pi)^{1/2} t^{3/2}} e^{-{a^2}/({2 t})} e^{-\lambda t} dt = e^{-a (2 \lambda)^{1/2} } \;.\label{eq:levy}
\end{equation}
For $a = 1$, and $\theta = (2 \lambda)^{1/2}$, this can be written as 
\begin{equation}
  E [\exp\{-\theta^2/(2G)\}] = \exp(-\theta),\quad G \sim \GammaRV(1/2, 1/2) 
  \label{eq:gamma}
\end{equation}
\end{proposition}

\begin{proof}
First substitute $t^{-1} = x^2$, which makes the left hand side in
\eqref{eq:levy} equal to 
\[
  \int_{0}^{\infty} \frac{a}{(2 \pi)^{1/2} t^{3/2}} e^{-{a^2}/({2 t})} e^{-\lambda t} dt = \left(\frac{2}{\pi}\right)^{1/2}ae^{-a (2 \lambda)^{1/2}} 
  \int_0^{\infty} e^{-({2}^{-1/2} ax - \lambda x^{-1})^2} dx = e^{-a (2 \lambda)^{1/2}}
  \;.
\]
The last step follows from \CS{} formula.  The second relationship \eqref{eq:gamma} follows by fixing $a = 1$, $\theta = (2\lambda)^{1/2}$ and
substituting $t = x^{-1}$
\[
\int_{0}^{\infty} \frac{a}{(2 \pi)^{1/2} t^{3/2}} 
e^{-{a^2}/({2 t})} e^{-\lambda t} dt 
= \frac{1}{(2 \pi)^{1/2}} \int_{0}^{\infty} e^{-{\theta^2}/({2x})} 
x^{-1/2} e^{-x/2} dx.
\]
The left hand side can be identified as 
$E\left\{e^{-\theta^2 / (2 G) } \right\}$ for 
$G \sim \GammaRV(1/2, 1/2)$. 
\end{proof}

\subsection{Logit and quantile as global-local mixtures}
Logistic modeling can be viewed within the global-local mixture framework via the \PG{} distribution \citep{polson_bayesian_2013}. This leads to efficient Markov chain Monte Carlo algorithms for inference. 
\begin{proposition}
The two key marginal distributions for the hyperbolic generalized inverse Gaussian \citep{barndorff1982normal} and \PG{} mixtures are
\begin{align}
 \frac{\alpha^2 - \kappa^2}{2\alpha} e^{-\alpha|x-\mu| + \kappa (x-\mu)} &= \int_0^{\infty} \phi(x \mid \mu + \kappa \lambda, \lambda) p_{\mathrm{GIG}}\left\{ \lambda \mid 1,0, (\alpha^2 - \kappa^2)^{1/2}\right\} d\lambda, \; \alpha \geq \kappa \geq 0, \label{eq:GIG}\\
\frac{1}{B(\alpha,\kappa)} \frac{e^{\alpha (x-\mu)}}{(1+e^{x-\mu})^{\alpha + \kappa}}&= \int_0^{\infty} \phi(x \mid \mu + \kappa \lambda, \lambda)p_{\mathrm{Polya}}(\lambda \mid \alpha,\kappa)  d\lambda\;, \label{eq:polya}
\end{align}
where $\phi(\mu + \kappa \lambda, \lambda)$ denotes the normal density function with mean $(\mu + \kappa \lambda)$ and variance $\lambda$.  The functions $p_{\mathrm{GIG}}$ and $p_{\mathrm{Polya}}$ are the corresponding local mixture densities for the generalized inverse Gaussian and the \PG{}, respectively. The logit and quantile identities can be derived using \CS{} identity. 
\end{proposition}
\begin{proof}
Let $f(x) = e^{-x^2/2}$, $a = \alpha$ and $b = |x-\phi|$ in \eqref{eq:identity}. Then,
\[
(2/\pi)^{1/2} \int_{0}^{\infty} 
\exp\left\{-\half \left(\alpha y - \frac{|x-\mu|}{y} \right)^2 \right\} dy 
= \frac{1}{\alpha} (2\pi)^{-1/2} \int_0^{\infty} e^{-\half y^2} dy 
= \frac{1}{\alpha}
\;.
\]
Let $\nu = y^2$. Rearranging the constant terms yields
\[
\frac{1}{\alpha} e^{-\alpha|x-\mu|} = \frac{1}{(2 \pi \nu)^{1/2}} \int_{0}^{\infty} \exp\left[-\left\{ \frac{(x-\mu)^2}{2\nu} + \frac{\alpha^2}{2} \nu \right\} \right]
d\nu
\;.
\]
Multiplying by $2^{-1}(\alpha^2-\kappa^2) e^{\kappa(x-\mu)}$ and completing the square yields
\begin{equation*}
  \frac{\alpha^2-\kappa^2}{2\alpha} \exp\left\{-\alpha|x-\mu| + \kappa(x-\mu)\right\} 
  = \int_0^{\infty} \phi(x \mid \mu + \kappa \nu, \nu) 
  \frac{\alpha^2-\kappa^2}{2} \exp\left(-\frac{\alpha^2-\kappa^2}{2} \nu \right) d \nu. 
\end{equation*}
The mixing distribution is exponential with rate parameter $(\alpha^2-\kappa^2)/2$, a special case of the generalized inverse Gaussian distribution introduced by Etienne Halphen circa 1941
\citep{seshadri1997halphen}.  The density with parameters $(\lambda, \delta, \gamma)$ 
has the form 
\begin{equation*}
  p_{\mathrm{GIG}}(x \mid \lambda, \delta, \gamma) = \frac{(\gamma/\delta)^{\lambda}}{2 K_{\lambda}(\delta \gamma)} x^{\lambda-1} 
  \exp\left\{ -\half (\delta^2 x^{-1} + \gamma^2 x )\right\}, \quad x, \lambda, \delta > 0,\;  p \in \Re
  \;,
\end{equation*}
where $K_{\lambda}$ is the modified Bessel function of the second kind.  The Liouville formula can be used to show that the above is a valid probability density
function.  When $\delta$ or $\gamma$ is zero, the normalizing constant takes the limiting values given by $K_{\lambda}(u) \asymp \Gamma(|\lambda|) 2^{|\lambda|-1} u^{|\lambda|}$ 
for $\lambda > 0$.  If $\delta=0$, the generalized inverse Gaussian is identical to a gamma distribution:
\[
p_{\mathrm{GIG}}(x \mid \lambda, \delta = 0 , \gamma) 
= \frac{\alpha^{\lambda}}{\Gamma(\lambda)} x^{\lambda-1} \exp(-\alpha x), 
\quad x > 0,\; \alpha = \gamma^2 / 2.
\]
%
We now present a simple proof for the \PG{} mixture in \eqref{eq:polya}. 
First, write $\kappa$ for $a-b/2$: 
\begin{equation}
  \frac{(e^{\psi})^a}{(1+e^{\psi})^b} = 2^{-b} e^{\kappa \omega} 
  \int_0^{\infty} e^{-\omega \psi^2/2} p(\omega) d\omega
  \;, 
  \label{eq:pg}
\end{equation}
where $\omega \sim \operatorname{PG}(b,0)$, a \PG{} random variable with density
$$
p(\omega \mid b, 0) = \frac{2^{b-1}}{\Gamma(b)} 
\sum_{n=0}^{\infty} (-1)^n \frac{\Gamma(n+b)}{\Gamma(n+1)} 
\frac{2n + b}{(2 \pi)^{1/2} \omega^{3/2}} 
\exp\left\{-\frac{(2 n + b)^2}{8 \omega} \right\}.
$$
The logit function corresponds to $a=0,b=1$ in \eqref{eq:pg}. \CS{} identity yields
\begin{equation}
  \frac{1}{1+e^{\psi}} = \half e^{- \psi/2} \int_0^{\infty} e^{-(\psi^2\omega)/2} p(\omega) d\omega,
  \quad p(\omega) = \sum_{n=0}^{\infty} (-1)^n \frac{2n+1}{ (2 \pi \omega^3)^{1/2}} 
  e^{-(2n+1)^2/(8 \omega)}
  \label{eq:logit}
  \;.
\end{equation}
To show \eqref{eq:logit}, write the right-hand side interchanging the integral and summation:
\begin{align*}
  I & = \half e^{-\psi/2} \sum_{n=0}^{\infty} (-1)^n \frac{2n+1}{(2 \pi)^{1/2}} 
  \int_0^{\infty} 
  \exp\left[-\left\{ \frac{\psi^2}{2} \omega + \frac{(2n+1)^2}{8 \omega} \right\} \right] \frac{1}{\omega^{3/2}} d\omega
  \;. 
\end{align*}
Using the change of variable $\omega = t^{-2}$ gives
\begin{align*}
  I & = \sum_{n=0}^{\infty} (-1)^n e^{-(n+1)\psi} 
  \frac{2n + 1}{(2 \pi)^{1/2}} 
  \left( \int_{0}^{\infty} 
    \exp\left[-\half \left\{ \frac{(2n+1)t}{2} - \frac{\psi}{t}\right\}^2 \right] dt 
  \right)
  \;.
\end{align*}
Applying the \CS{} identity to the inner integral yields 
\[
\int_{0}^{\infty} 
\exp\left[-\half \left\{ \frac{(2n+1)t}{2} - \frac{\psi}{t}\right \}^2 \right] dt 
= \int_0^{\infty} \frac{e^{-y^2/2}}{2n+1} dy= \frac{(2\pi)^{1/2}}{2n+1}
\;,
\]
which implies $I = \sum_{n=0}^{\infty} (-1)^n \exp\{-(n+1) \psi\} = \{1+\exp(\psi)\}^{-1}$. 
\end{proof}

\begin{remark}
When $\alpha = \kappa$, we have the limiting result $(\alpha^2-\kappa^2)^{-1} p_{\mathrm{GIG}}\{1,0, (\alpha^2-\kappa^2)^{1/2} \} = 1,$
or equivalently in terms of densities, with a marginal improper uniform prior, $p(\lambda) = 1$,
\begin{equation}
  \int_{0}^{\infty} \phi(b \mid -a\lambda, c\lambda) d\lambda = a^{-1} \exp\left\{-2 \max(ab/c,0)\right\}
  \;. 
  \label{eq:svm}
\end{equation}
This pseudo-likelihood represents support vector machines as a global-local mixture. The identity for quantile regression, which is a limiting case of the above identities by applying Fatou-Lebesgue theorem, is the following: 
\[
c^{-1}\exp\{ 2c^{-1} \rho_q(b) \}= \int_{0}^{\infty} \phi( b \mid \lambda - 2\tau \lambda, c \lambda){\rm e}^{-2\tau(1-\tau)\lambda} d\lambda, \quad c, \tau > 0,
\]
where $\rho_q(b) = \rvert b \lvert / 2 + (q-1/2) b$ is the check-loss function \citep{polson_data_2013}.
\end{remark}

\citet{polson_data_2011} derive this as a direct consequence of the lasso identity 
\[
\int_0^{\infty} p/(2 \pi \lambda)^{1/2} \exp\left\{-\left(p^2 \lambda+q^2 \lambda^{-1}\right)/2\right\} d\lambda = e^{-\lvert pq \rvert}.
\]
Applying the Liouville identity yields
\[
\int_{0}^{\infty} f\left(ax + \frac{b}{x} \right) x^{-1/2} dx = a^{-1/2} \int_{0}^{\infty} f\left\{ 2 (ab)^{1/2} + y \right\} y^{-1/2} dy, \quad a, b > 0.
\]
Setting $f(x) = e^{-x}$, $a = p^2/2$, and $b = q^2/2$ we get
\begin{align*}
  \int_0^{\infty} \frac{e^{-(p^2 \lambda + q^2 \lambda^{-1})/2}}{\lambda^{1/2}} d\lambda
  & = \frac{2^{1/2}}{p} \int_0^{\infty} e^{-|pq| + y} y^{-1/2} d y \\
  & = \frac{2^{1/2} e^{-|pq|}}{p} \int_0^{\infty} e^{-y} y^{-1/2} d y 
  = \frac{(2\pi)^{1/2} e^{-|pq|}}{p}
  \;.
\end{align*}

\citet{hans2011comment} shows that the elastic-net regression can be recast as a
global-local mixture with a mixing density belonging to the orthant-normal
family of distributions.  The orthant-normal prior on a single regression
coefficient, $\beta$, given hyper-parameters $\lambda_1$ and $\lambda_2$, 
has a density function with the following form:
\begin{equation}
  p(\beta \mid \lambda_1, \lambda_2)  = 
  \begin{cases} 
   \phi(\beta \mid \frac{\lambda_1}{2\lambda_2}, \frac{\sigma^2}{\lambda_2}) 
   / 2\Phi\left(-\frac{\lambda_1}{2\sigma \lambda_2^{1/2} }\right), & \quad \beta < 0, 
   \\
   \phi(\beta \mid \frac{-\lambda_1}{2\lambda_2}, \frac{\sigma^2}{\lambda_2}) / 
   2\Phi\left(-\frac{\lambda_1}{2\sigma \lambda_2^{1/2} }\right), & \quad \beta \geq 0.
  \end{cases} 
  \;
  \label{eq:hans}
\end{equation}

\section{Convolution mixtures}
\label{sec:convolutions}

Another interesting area of application is convolution mixtures and marginal densities for location-scale mixture problems. We show that the Cauchy
convolution \citep{pillai2015unexpected} and inverse-gamma convolution can be derived similarly \citep{polson_halfcauchy_2012}. \citet{bhadra_default_2016} shows that the regularly varying tails of half-Cauchy priors work well for low-dimensional functions of normal
vector mean, where flat priors give poorly calibrated inference. 
\begin{lemma}
  Let $X_i \sim \CauchyRV(0,1)$ $(i = 1, 2)$ be Cauchy distributed random variates, then $Z = w_1 X_1 + w_2 X_2 \sim \CauchyRV( 0, w_1 + w_2).$ where $w_1,w_2 > 0$.
\end{lemma}
\begin{lemma}
  Let $X_i \sim \InvGaussRV(\alpha t_i, \alpha t_i^2)$ $(i = 1, 2)$, then $Z = X_1 + X_2 \sim \InvGaussRV\{\alpha (t_1 + t_2), \alpha (t_1^2+t_2^2)\},$ where $\alpha, t_1, t_2 \geq 0$, and $\InvGaussRV(\alpha t, \alpha t^2)$ is an inverse-Gaussian random variable with density
\[
    f(x) = \frac{t \alpha^{1/2} e^t}{(2 \pi)^{1/2} x^{3/2}} 
    \exp\left( -\frac{\alpha t^2}{2x} - \frac{x}{2\alpha} \right), \quad x \geq 0.
\]
\end{lemma}

Both of these results follow from straightforward applications of the \CS{} 
transformation. We give a proof for the Cauchy convolution identity below.

\begin{proof}
Exploiting symmetry and the Lagrange identity $(a^2 + b^2)(c^2 + d^2) = (ac+bd)^2 + (ad-bc)^2,$ leads to the convolution density
\begin{align*} 
  f_Z(z) &= 2 \int_{0}^{\infty} 
  \frac{1}{ \pi w_1 (1+ x^2/w_1^2)} \frac{1}{\pi w_2 \{1+ (z-x)^2 / w_2^2 \} } dx
  \\
  & = \frac{2}{\pi^2 w_1 w_2} \int_{0}^{\infty} 
  \frac{1}{\{1+ w_1^{-1} w_2^{-1} x (z-x) \}^2 + \{w_2^{-1}z - (w_1^{-1}+ w_2^{-1}) x \}^2 } dx.
\end{align*}
Transforming $x$ to $x + w_2^{-1}z (w_1^{-1} + w_2^{-1})^{-1}$ and letting $a = 1 + z^2(w_1+w_2)^{-2}$, $b =(w_1 w_2)^{-1}$, $c = z (w_2-w_1) \{(w_1+w_2) w_1 w_2\}^{-1}$, $d = z (w_2-w_1)\{(w_1+w_2) w_1 w_2\}^{-1}$ gives
\begin{align*}
  f_Z(z) &= \frac{2}{\pi^2 w_1 w_2} \int_{0}^{\infty} 
  \left[ 
    \left\{ 1 + \frac{z^2}{(w_1+w_2)^2} - \frac{x^2}{w_1w_2} + 
      xz \frac{w_2-w_1}{(w_1+w_2) w_1 w_2} \right\}^2 + 
    x^2 \left(\frac{w_1 + w_2}{w_1w_2} \right)^2 
  \right]^{-1} dx 
  \\
  &= \frac{2}{\pi^2 w_1 w_2} \int_{0}^{\infty} 
  \frac{dx}{\left( a - b x^2 + cx \right)^2 + x^2 d^2} 
  = \frac{2}{\pi^2 w_1 w_2} \int_{0}^{\infty} 
  \frac{dx/x^2}{\left(a/x - bx + c \right)^2 + d^2 }.
\end{align*}
If we let $y = x^{-1}$ and apply the \CS{} transformation, we arrive at 
\[
  f_Z(z) = \frac{2}{\pi w_1 w_2} \int_{0}^{\infty} \frac{dy}{2a (y^2 + d^2)} 
  = \frac{1}{\pi w_1 w_2} \frac{1}{ad}= \frac{1}{\pi (w_1+w_2)} \frac{1}{1+ z^2/(w_1+w_2)^2}.
\]
A simple induction argument proves that the sum of any number of independent
Cauchy random variates is also another Cauchy.
\end{proof}

One can also use the characteristic function of $X \sim \CauchyRV(\mu, \sigma)$, $\psi_X(t) = \exp(it \mu - |t| \sigma^2)$, and the relation $\psi_{X+Y}(t) = \psi_X(t) \psi_Y(t)$ to derive the result in just one step. For $X = \sum_{i=1}^{p} \omega_i C_i$ and $C_i \sim \CauchyRV(0,1)$, when $\sum_{i=1}^{p} \omega_i = 1$ we have 
$\phi_X(t) = \exp\left(-\sum_{i=1}^{p}\omega_i |t|\right) = \exp(-|t|) = \phi_C(t),$ where $C \sim \CauchyRV(0, 1)$. 

The most general result in this category is due to \cite{pillai2015unexpected},
who they showed the following: 
Let $(X_1,\ldots,X_m)$ and $(Y_1, \ldots, Y_m)$ be independent and
identically distributed $\NormRV(0, \Sigma)$ for an arbitrary
positive definite matrix $\Sigma$, then 
$Z = \sum_{j=1}^{m} w_j X_j / Y_j \sim \CauchyRV(0, 1)$, 
as long as $(w_1, \ldots, w_m)$ is independent of $(X, Y)$,
$w_j \geq 0\ (j = 1, \ldots, m)$ and $\sum_{j=1}^{m} w_j = 1$. 

\section{Discussion}
\label{sec:discussion}

The \CS{} and Liouville transformations not only guarantee an simple normalizing constant for $f(\cdot)$, it also establishes the wide class of unimodal densities as global-local scale mixtures. Global-local scale mixtures that are conditionally Gaussian hold a special place in statistical modeling and can be rapidly fit using an expectation-maximization algorithm, as
pointed out by \cite{polson_data_2013}. \cite{palmer_amica:_2011} provides a similar tool for modeling multivariate dependence by writing general non-Gaussian multivariate densities as multivariate Gaussian scale mixtures. 

We end our paper with conjectures that two other remarkable identities arise as corollaries of such transformation identities. The first one is a recent result by \cite{zhang2014uniform} that proves a uniform correlation mixture of a bivariate Gaussian density with unit variance is a function of the maximum norm: 
\begin{equation}
  \int_{-1}^{1} \frac{1}{4 \pi (1-\rho^2)^{1/2} } 
  \exp\left\{ - \frac{x_1^2 + x_2^2 - 2 \rho x_1 x_2}{2 (1-\rho^2)} \right\} d\rho = 
  \half \left\{1- \Phi(\vectornorm{x}_{\infty})\right\} 
  \;, 
  \label{eq:bivar}
\end{equation}
where $\Phi(\cdot)$ is the standard normal distribution function and $\vectornorm{x}_{\infty} = \max\{ x_1, x_2\}$. The bivariate density on the
right side of \eqref{eq:bivar} was introduced by \citet{bryson1982constructing} as uniform mixtures of a chi random variate with 3 degrees of freedom, but the representation as a uniform correlation mixture is a new find.  We make a few remarks connected to the Erdelyi's integral identity, which is key to the proof of the uniform correlation mixture of \eqref{eq:bivar}. 
\begin{lemma}
Erdelyi's identity, defined by
\begin{equation}
  \int_{1/2}^{\infty} \frac{e^{-x^2 z}}{4 \pi z 	(2z-1)^{1/2}} dz = \half \left\{1-\Phi(x)\right\}, \quad x \geq 0, \label{eq:erdelyi}
\end{equation}
follows from the Laplace transformation $(1+u)^{-1} = \int_0^{\infty} \exp\{-v(1+u)\} dv$. 
\end{lemma}

\begin{proof}
  Apply the transform $u = 2z-1$ to the left hand side of \eqref{eq:erdelyi}, denoted by $I$, to obtain 
  \[
  I = \int_{0}^{\infty} \frac{e^{-x^2/\{2(1+u)\}}}{4 \pi {u}^{1/2} (1+u)} du \;.
  \]
  Using the Laplace transformation $(1+u)^{-1} = \int_0^{\infty} e^{-v(1+u)} dv$ yields
  \begin{align*}
    I &= \int_{0}^{\infty} \frac{e^{-x^2/\{2(1+u)\}}}{4 \pi {u}^{1/2}} 
    \int_0^{\infty} e^{-v(1+u)} dv du 
    = \int_{v= 0}^{\infty} \int_{u=0}^{\infty} 
    \frac{e^{-({x^2}/{2} + v)(1+u)}}{4 \pi {u}^{1/2}} dv du
    \\
    &= \int_{v= 0}^{\infty} \frac{1}{4\pi} e^{-({x^2}/{2} + v)} 
    \int_{u=0}^{\infty} u^{-1/2} e^{-({x^2}/{2} + v) u} du dv 
    = \int_{v= 0}^{\infty} \frac{e^{-(x^2 + 2v)/2}}{2 (2\pi)^{1/2}} 
    \frac{1}{(x^2+ 2v)^{1/2} } dv
    \;,
  \end{align*}
  and letting $z^2 = x^2 + 2v$ we get 
  \begin{equation*}
    I = \half \int_{z = |x|}^{\infty} \frac{1}{(2\pi)^{1/2}} e^{-z^2/2} dz 
    = \half \left\{1 - \Phi(|x|)\right\} \;.
  \end{equation*} 
\end{proof}



The second candidate is the symmetric stable distribution, defined by its characteristic function $\phi(t) = \exp( -|t|^{\alpha}), 0 < \alpha \leq 2$. It admits a normal scale mixture representation with mixing density as $f(v) = 2^{-1} s_{\alpha/2}(v/2), v > 0$, where $s_{\alpha/2}$ is the positive stable density with index $\alpha / 2$ \citep{gneiting1997normal}. The exponential power density arising as a dual of the symmetric stable density also has a normal scale mixture representation with important application in
Bayesian bridge regression \citep{polson_bayesian_2014}.
\[
e^{-|x|^\alpha} = \int_0^{\infty} e^{-x\eta} g(\eta) d\eta, \quad g(\eta) = \sum_{j=1}^{\infty} (-1)^j \frac{\eta^{-j \alpha-1}}{j! \Gamma(-\alpha j)}
  \;,
\]
\citet{polson_bayesian_2014} derive this as a limiting result of the scale-mixture of beta representation for $k$-montone densities and utilizing the complete monotonicity of exponential power density. Regularization, in this case, is an outcome of a normal scale mixture with respect to an $\alpha$-stable random variable.  We conjecture that these two results follow from the \CS{} formula \eqref{eq:identity}. Other potential applications include using Liouville formula to recognize and generate global-local mixtures, and to calculate higher-order closed-form moments $E(X^n)$ for random variables $X$ that admit a global-local representation. 

\bibliographystyle{biometrika}
\bibliography{glref}

\end{document}